\documentclass[a4paper,12pt]{article}

\usepackage{authblk}
\usepackage{parskip}
\usepackage{mathtools}
\usepackage{amssymb}
\usepackage{amsmath}
\usepackage{latexsym}
\usepackage{mathrsfs}  
\usepackage{bbm}       
\usepackage{amsthm}    
\usepackage{relsize}   
\usepackage{graphicx}
\usepackage{thmtools}  
\usepackage{mathrsfs}  
\usepackage{paralist}  
\usepackage{lipsum}    
\usepackage[all]{xy}   

\numberwithin{equation}{section}

\usepackage{titlesec}   

\titleformat{\section}[block]{\bfseries\filcenter}
{{\upshape\thesection\enspace}}{.5em}{}

\titleformat{\subsection}[block]{\filcenter}
{{\upshape\thesubsection\enspace}}{.5em}{} 

\titleformat{\subsubsection}[block]{\filcenter}
{{\upshape\thesubsubsection\enspace}}{.5em}{} 

\usepackage{enumitem}  
\setlist{nosep}  
\setitemize[0]{leftmargin=*}
\setenumerate[0]{leftmargin=*}
\setenumerate[1]{label={(\arabic*)}} 

\newcommand{\N}{\mathbb{N}}     
\newcommand{\R}{\mathbb{R}}     
\newcommand{\Prob}{\mathbb{P}}  
\newcommand{\Exp}{\mathbb{E}}   
\newcommand{\goth}[1]{\mathfrak{#1}} 
\newcommand{\ind}[2]{\mathbbm{1}_{#1}\left( #2 \right)}          
\newcommand{\inner}[2]{\left\langle #1 \, , \, #2 \right\rangle} 
\newcommand{\norm}[1]{\left|\left|#1\right|\right|}              
\newcommand{\triplet}[3]{\left( #1, #2, #3 \right) }             
\newcommand{\ProbSpace}{\triplet{\Omega}{\mathscr{F}}{\Prob}}    
\newcommand{\abs}[1]{\left| #1 \right|}                          

\newcommand{\defeq}{\mathrel{\mathop:}=}                         
\newcommand\restr[2]{{
  \left.\kern-\nulldelimiterspace 
  #1 
  \vphantom{\big|} 
  \right|_{#2} 
  }}
  
\makeatletter
\newsavebox{\@brx}
\newcommand{\llangle}[1][]{\savebox{\@brx}{\(\m@th{#1\langle}\)}%
  \mathopen{\copy\@brx\kern-0.5\wd\@brx\usebox{\@brx}}}
\newcommand{\rrangle}[1][]{\savebox{\@brx}{\(\m@th{#1\rangle}\)}%
  \mathclose{\copy\@brx\kern-0.5\wd\@brx\usebox{\@brx}}}
\makeatother

\theoremstyle{plain} 
\newtheorem{theorem}{Theorem}[section]

\newtheorem{lemma}[theorem]{Lemma}
\newtheorem{assumption}[theorem]{Assumption}

\theoremstyle{definition} 

\newtheorem{example}[theorem]{Example}

 \title{Time regularity of stochastic convolutions and stochastic evolution equations in duals of nuclear spaces}
\author{C. A. Fonseca-Mora}
\affil{  Escuela de Matem\'{a}tica,  Universidad de Costa Rica,\\ San Jos\'{e}, 11501-2060, Costa Rica. \\

\noindent E-mail:  christianandres.fonseca@ucr.ac.cr }

\date{}

\begin{document}

 \maketitle

\abstract{ 
Let $\Phi$ a locally convex space and $\Psi$ be a quasi-complete, bornological, nuclear space  (like spaces of smooth functions and distributions) with dual spaces $\Phi'$ and $\Psi'$. In this work we introduce sufficient conditions for time regularity properties of the $\Psi'$-valued stochastic convolution $\int_{0}^{t} \int_{U} S(t-r)'R(r,u) M(dr,du)$, $t \in [0,T]$, where $(S(t): t \geq 0)$ is a $C_{0}$-semigroup on $\Psi$, $R(r,\omega,u)$ is a suitable operator form $\Phi'$ into $\Psi'$ and $M$ is a  cylindrical-martingale valued measure on $\Phi'$. Our result is latter applied to study time regularity of solutions to $\Psi'$-valued stochastic evolutions equations. 
}

\smallskip



\section{Introduction}

Let $\Phi$ be a locally convex space with strong dual $\Phi'$. Let $M=(M(t,A): t \geq 0, A \in \mathcal{R})$ be a cylindrical martingale-valued measure on $\Phi'$, i.e. $(M(t,A): t \geq 0)$ is a cylindrical martingale in $\Phi'$ for each $A \in \mathcal{R}$ and $M(t,\cdot)$ is finitely additive on a $\mathcal{R}$ for each $t \geq 0$. Here $\mathcal{R} \subseteq \mathcal{B}(U)$ is a ring and $U$ a topological space. 
Moreover, let $\Psi$ be a quasi-complete bornological nuclear space, $Z_{0}$ a $\Psi'$-valued random variable, and let $A$ be the generator of a $S_{0}$-semigroup $(S(t): t \geq 0)$ of continuous linear operators on $\Psi$. In this paper we study time regularity of solutions to the following class of stochastic evolution equations 
\begin{equation} \label{generalSEEIntro}
d X_{t}= (A'X_{t}+ B (t,X_{t})) dt+\int_{U} F(t,u,X_{t}) M (dt,du), 
\end{equation}
with initial value $X_{0}=Z_{0}$, where $A'$  is the dual operator to $A$ and the coefficients $B$ and $F$ satisfy appropriate measurability conditions. 

In \cite{FonsecaMora:2018-1}, sufficient growth and Lipschitz  conditions on the coefficients has been introduced to show existence and uniqueness of a weak solution to \eqref{generalSEEIntro}. In particular, this weak solution is given by the mild or evolution solution: for every $ t \geq 0$, $\Prob$-a.e.
\begin{equation} \label{equationMildSolutionIntro}
X_{t} = S(t)'Z_{0}+ \int^{t}_{0} S(t-r)' B(r,X_{r})dr  +  \int^{t}_{0} \int_{U} S(t-r)' F(r,u,X_{r}) M(dr,du).
\end{equation}
The second integral at the right-hand side of \eqref{equationMildSolutionIntro} is the stochastic convolution of the dual semigroup $(S(t)': t \geq 0)$ and the coefficient $F$. This stochastic integral is defined by means of the theory of (strong) stochastic integration with respect to the cylindrical martingale-valued measure $M$ as defined in \cite{FonsecaMora:2018-1}. 

In order  to show that the mild solution \eqref{equationMildSolutionIntro} has a continuous or a c\`{a}dl\`{a}g version, we will require to show first that for $R$ integrable with respect to $M$ (see Section \ref{subsecStocInteg}), the stochastic convolution  $ \int^{t}_{0} \int_{U} S(t-r)' R(r,u) M(dr,du)$ possesses a continuous or a c\`{a}dl\`{a}g version under mild assumptions on the $C_{0}$-semigroup. 

To be more specific, let  us assume $( S(t) : t \geq 0)$ is a $(C_{0},1)$-semigroup of continuous linear operators on $\Psi$ (see \cite{Babalola:1974}). According to Theorem 2.8 in \cite{Babalola:1974} there exists a family $\Pi$ of seminorms  generating the topology on $\Psi$ such that for each $p \in \Pi$ there exist $\theta_{p} \geq 0$ such that 
\begin{equation}\label{eqDefC01Semigroup}
 p(S(t)\psi) \leq e^{\theta_{p} t} p(\psi), \quad \mbox{for all } t \geq 0, \, \psi \in \Psi.
\end{equation}
In this work we will say that  $(S(t):t\geq 0)$ is a \emph{Hilbertian $(C_{0},1)$-semigroup} if such a family $\Pi$ can be chosen such that each $p \in \Pi$ is a Hilbertian seminorm. 

In Section \ref{sectRegulaPathsStochConvo}  we will show (see Theorem \ref{theoCadlagVerStochConv}) that if $(S(t):t\geq 0)$ is a Hilbertian $(C_{0},1)$-semigroup then the stochastic convolution  $ \int^{t}_{0} \int_{U} S(t-r)' R(r,u) M(dr,du)$ possesses a continuous or a c\`{a}dl\`{a}g version if either $M$ has respectively (as a cylindrical process) continuous or c\`{a}dl\`{a}g paths. We will indeed show that this version 
takes values in a Hilbert space $\Psi'_{p}$ continuously embedded in $\Psi'$.

In the context of a Hilbert space with norm $p$ a $C_{0}$-semigroup satisfying \eqref{eqDefC01Semigroup} is often referred to as \emph{exponentially bounded} and it is known from the work of Kotelenez \cite{Kotelenez:1982} that the stochastic convolution for such a semigroup and with respect to a Hilbert-space valued square integrable martingale possesses a continuous or a c\`{a}dl\`{a}g. The work we carried out in Section \ref{sectRegulaPathsStochConvo} can be thought a natural extension of the work of Kotelenez to the context of stochastic convolution in the dual of a nuclear space. 

To the extent of our knowledge only the works \cite{FonsecaMora:SSELevy, PerezAbreuTudor:1992} studied time regularity of stochastic convolutions with values in the dual of a nuclear space. In each of these works the integrand $R$ is trivial (i.e. the identity operator) and the convolution is defined with respect to a $\Phi'$-valued processes for a reflexive nuclear space $\Phi$. Below we briefly describe the results in \cite{FonsecaMora:SSELevy, PerezAbreuTudor:1992}. 

Time regularity of stochastic convolutions of the form $\int_{0}^{t} S(t-r)' dM_{r}$ has been studied in  (\cite{PerezAbreuTudor:1992}, Theorem 4.1)  under the assumption that $\Psi=\Phi$ is a Fr\'{e}chet nuclear space and $M$ is a $\Psi'$-valued square integrable martingale. In \cite{PerezAbreuTudor:1992} it is not explicitly assumed that $(S(t):t\geq 0)$ is a Hilbertian $(C_{0},1)$-semigroup but this property is used on their proofs. Indeed, it  is a common practice in the literature of stochastic analysis in duals of nuclear spaces that a $(C_{0},1)$-semigroup is assumed to be Hilbertian (see for example  \cite{Ding:1999, Wu:1994, Wu:1995}). 

In  (\cite{FonsecaMora:SSELevy}, Theorem 5.7) and under the assumption that both $\Phi$ and $\Phi$ are quasi-complete bornological nuclear spaces and $M$ is a $\Psi'$-valued L\'{e}vy process, it is shown that if the generator $A$ of $(S(t):t\geq 0)$  is a continuous operator, then $\int_{0}^{t} S(t-r)' dM_{r}$ has a c\`{a}dl\`{a}g version. It is worth to mention that in \cite{FonsecaMora:SSELevy} it is not assumed that the $(C_{0},1)$-semigroup $(S(t):t\geq 0)$ is Hilbertian. 
 
Now in Section \ref{sectTimeReguSEE} we return to our study of time regularity of solutions to stochastic evolution equations. Indeed, by applying our aforementioned result on time regularity of stochastic convolutions we show (see Theorem \ref{theoCadlagSolSEEwithCMVM}) that under some growth and Lipschitz conditions introduced in \cite{FonsecaMora:2018-1} the solution \eqref{equationMildSolutionIntro} to \eqref{generalSEEIntro} has a unique continuous or c\`{a}dl\`{a}g version taking values and having square moments in a Hilbert space $\Psi'_{p}$ continuously embedded in $\Psi'$ for each bounded interval of time.  

Finally, in Section \ref{sectAppliExample} we provide some examples and applications of Hilbertian $(C_{0},1)$-semigroups and of stochastic evolution equations for which our results can be applied to show the existence of continuous or c\`{a}dl\`{a}g solutions. In particular, equations driven by L\'{e}vy noise will be considered.

\section{Preliminaries}\label{sectPrelim}

\subsection{Locally convex spaces and linear operator}

Let $\Phi$ be a locally convex space (we will only consider vector spaces over $\R$). $\Phi$ is \emph{quasi-complete} if each bounded and closed subset of it is complete. $\Phi$ is called  \emph{bornological} (respectively \emph{ultrabornological}) if it is the inductive limit of a family of normed (respectively Banach)  spaces. A \emph{barreled space} is a locally convex space such that every convex, balanced, absorbing and closed subset is a neighborhood of zero. For further details see \cite{Jarchow, Schaefer, Treves}.   

If $p$ is a continuous semi-norm on $\Phi$ and $r>0$, the closed ball of radius $r$ of $p$ given by $B_{p}(r) = \left\{ \phi \in \Phi: p(\phi) \leq r \right\}$ is a closed, convex, balanced neighborhood of zero in $\Phi$. A continuous seminorm (respectively norm) $p$ on $\Phi$ is called \emph{Hilbertian} if $p(\phi)^{2}=Q(\phi,\phi)$, for all $\phi \in \Phi$, where $Q$ is a symmetric, non-negative bilinear form (respectively inner product) on $\Phi \times \Phi$. For any given continuous seminorm $p$ on $\Phi$ let $\Phi_{p}$ be the Banach space that corresponds to the completion of the normed space $(\Phi / \mbox{ker}(p), \tilde{p})$, where $\tilde{p}(\phi+\mbox{ker}(p))=p(\phi)$ for each $\phi \in \Phi$. If $\Phi_{p}$ is separable then we say that $p$ is \emph{separable}. We denote by  $\Phi'_{p}$ the Banach space dual to $\Phi_{p}$ and by $p'$ the corresponding dual norm. Observe that if $p$ is Hilbertian then $\Phi_{p}$ and $\Phi'_{p}$ are Hilbert spaces. If $q$ is another continuous seminorm on $\Phi$ for which $p \leq q$, we have that $\mbox{ker}(q) \subseteq \mbox{ker}(p)$ and the inclusion map from $\Phi / \mbox{ker}(q)$ into $\Phi / \mbox{ker}(p)$ has a unique continuous and linear extension that we denote by $i_{p,q}:\Phi_{q} \rightarrow \Phi_{p}$. Furthermore, we have the following relation: $i_{p}=i_{p,q} \circ i_{q}$.

Let $p$ and $q$ be continuous Hilbertian semi-norms on $\Phi$ such that $p \leq q$.
The space of continuous linear operators (respectively Hilbert-Schmidt operators) from $\Phi_{q}$ into $\Phi_{p}$ is denoted by $\mathcal{L}(\Phi_{q},\Phi_{p})$ (respectively $\mathcal{L}_{2}(\Phi_{q},\Phi_{p})$) and the operator norm (respectively Hilbert-Schmidt norm) is denote by $\norm{\cdot}_{\mathcal{L}(\Phi_{q},\Phi_{p})}$ (respectively $\norm{\cdot}_{\mathcal{L}_{2}(\Phi_{q},\Phi_{p})}$).

We denote by $\Phi'$ the topological dual of $\Phi$ and by $\inner{f}{\phi}$ the canonical pairing of elements $f \in \Phi'$, $\phi \in \Phi$. Unless otherwise specified, $\Phi'$ will always be consider equipped with its \emph{strong topology}, i.e. the topology on $\Phi'$ generated by the family of semi-norms $( \eta_{B} )$, where for each $B \subseteq \Phi$ bounded we have $\eta_{B}(f)=\sup \{ \abs{\inner{f}{\phi}}: \phi \in B \}$ for all $f \in \Phi'$.  Recall that $\Phi$ is called \emph{semi-reflexive} if the canonical (algebraic) embedding of $\Phi$ into $\Phi''$ is onto, and is called \emph{reflexive} if the canonical embedding is indeed an isomorphism (of topological vector spaces).


Let us recall that a (Hausdorff) locally convex space $(\Phi,\mathcal{T})$ is called \emph{nuclear} if its topology $\mathcal{T}$ is generated by a family $\Pi$ of Hilbertian semi-norms such that for each $p \in \Pi$ there exists $q \in \Pi$, satisfying $p \leq q$ and the canonical inclusion $i_{p,q}: \Phi_{q} \rightarrow \Phi_{p}$ is Hilbert-Schmidt. Other equivalent definitions of nuclear spaces can be found in \cite{Pietsch, Treves}. Recall that any quasi-complete, bornological, nuclear space is barreled and semi-reflexive, therefore is reflexive (see Theorem IV.5.6 in \cite{Schaefer}, p.145).

The following are all examples of complete, ultrabornological (hence barrelled) nuclear spaces: 
the spaces of functions $\mathscr{E}_{K} \defeq \mathcal{C}^{\infty}(K)$ ($K$: compact subset of $\R^{d}$) and $\mathscr{E}\defeq \mathcal{C}^{\infty}(\R^{d})$, the rapidly decreasing functions $\mathscr{S}(\R^{d})$, and the space of test functions $\mathscr{D}(U) \defeq \mathcal{C}_{c}^{\infty}(U)$ ($U$: open subset of $\R^{d}$), as well are the spaces of distributions $\mathscr{E}'_{K}$, $\mathscr{E}'$, $\mathscr{S}'(\R^{d})$, and $\mathscr{D}'(U)$. Other examples are the space of harmonic functions $\mathcal{H}(U)$ ($U$: open subset of $\R^{d}$), the space of polynomials $\mathcal{P}_{n}$ in $n$-variables and the space of real-valued sequences $\R^{\N}$ (with direct product topology). For references see \cite{Pietsch, Schaefer, Treves}. 

Let $\Psi$ a locally convex space. 
A family $( S(t): t \geq 0) \subseteq \mathcal{L}(\Psi,\Psi)$  is  a \emph{$C_{0}$-semigroup} on $\Psi$ if \begin{inparaenum}[(i)] \item $S(0)=I$, $S(t)S(s)=S(t+s)$ for all $t, s \geq 0$, and \item $\lim_{t \rightarrow s}S(t) \phi = S(s) \psi$, for all $s \geq 0$ and any $\psi \in \Psi$. 
\end{inparaenum} The \emph{infinitesimal generator} $A$ of  $( S(t) : t \geq 0)$ is the linear operator on $\Psi$ defined by $ A \psi = \lim_{h \downarrow 0} \frac{S(h) \psi -\psi}{h}$ (limit in $\Psi$), whenever the limit exists; the domain of $A$ being the set $\mbox{Dom}(A) \subseteq \Psi$ for which the above limit exists.
If the space $\Psi$ is reflexive, then the family $( S(t)' : t \geq 0)$ of  dual operators is a $C_{0}$-semigroup on $\Psi'$ with generator $A'$, that we call the \emph{dual semigroup} and the \emph{dual generator} respectively.

A $C_{0}$-semigroup $( S(t) : t \geq 0)$ is called a \emph{$(C_{0},1)$-semigroup} if for each continuous seminorm $p$ on $\Psi$  there exists some $\vartheta_{p} \geq 0$ and a continuous seminorm $q$ on $\Psi$ such that  $ p(S(t)\psi) \leq e^{\vartheta_{p} t} q(\psi)$, for all $t \geq 0$ and $\psi \in \Psi$. If in the above inequality $\vartheta_{p}=\omega$ with $\omega$ a positive constant (independent of $p$) $( S(t) : t \geq 0)$ is called \emph{quasiequicontinuous}, and if $\omega=0$ $( S(t) : t \geq 0)$ is called \emph{equicontinuous}. It is worth to mention that even if $\Psi$ is reflexive, the dual semigroup $( S(t)' : t \geq 0)$ to a $(C_{0},1)$-semigroup $( S(t) : t \geq 0)$  is not in general a $(C_{0},1)$-semigroup on $\Psi'$ (see  \cite{Babalola:1974}, Section 6). However,  if $( S(t) : t \geq 0)$ is equicontinous and $\Psi$ is reflexive we have that $( S(t)' : t \geq 0)$ is equicontinuous (see \cite{Komura:1968}, Theorem 1 and its Corollary). 

\subsection{Cylindrical and stochastic processes}

Throughout this work we assume that $\ProbSpace$ is a complete probability space and consider a filtration $( \mathcal{F}_{t} : t \geq 0)$ on $\ProbSpace$ that satisfies the \emph{usual conditions}, i.e. it is right continuous and $\mathcal{F}_{0}$ contains all subsets of sets of $\mathcal{F}$ of $\Prob$-measure zero. We denote by $L^{0} \ProbSpace$ the space of equivalence classes of real-valued random variables defined on $\ProbSpace$. We always consider the space $L^{0} \ProbSpace$ equipped with the topology of convergence in probability and in this case it is a complete, metrizable, topological vector space. 
We denote by $\mathcal{P}_{\infty}$ the predictable $\sigma$-algebra on $[0, \infty) \times \Omega$ and for any $T>0$, we denote by $\mathcal{P}_{T}$ the restriction of $\mathcal{P}_{\infty}$ to $[0,T] \times \Omega$.  

Let $\Phi$ be a locally convex space. A \emph{cylindrical random variable}\index{cylindrical random variable} in $\Phi'$ is a linear map $X: \Phi \rightarrow L^{0} \ProbSpace$ (see \cite{FonsecaMora:2018}). If $X$ is a cylindrical random variable in $\Phi'$, we say that $X$ is \emph{$n$-integrable} ($n \in \N$)  if $ \Exp \left( \abs{X(\phi)}^{n} \right)< \infty$, $\forall \, \phi \in \Phi$, and has \emph{zero-mean} if $ \Exp \left( X(\phi) \right)=0$, $\forall \phi \in \Phi$. 

Let $X$ be a $\Phi'$-valued random variable, i.e. $X:\Omega \rightarrow \Phi'$ is a $\mathscr{F}/\mathcal{B}(\Phi')$-measurable map. For each $\phi \in \Phi$ we denote by $\inner{X}{\phi}$ the real-valued random variable defined by $\inner{X}{\phi}(\omega) \defeq \inner{X(\omega)}{\phi}$, for all $\omega \in \Omega$. The linear mapping $\phi \mapsto \inner{X}{\phi}$ is called the \emph{cylindrical random variable induced/defined by} $X$. We will say that a $\Phi'$-valued random variable $X$ is \emph{$n$-integrable} if the cylindrical random variable induced by $X$ is \emph{$n$-integrable}. 
 
Let $J=\R_{+} \defeq [0,\infty)$ or $J=[0,T]$ for  $T>0$. We say that $X=( X_{t}: t \in J)$ is a \emph{cylindrical process} in $\Phi'$ if $X_{t}$ is a cylindrical random variable for each $t \in J$. Clearly, any $\Phi'$-valued stochastic processes $X=( X_{t}: t \in J)$ \emph{induces/defines} a cylindrical process under the prescription: $\inner{X}{\phi}=( \inner{X_{t}}{\phi}: t \in J)$, for each $\phi \in \Phi$. 

If $X$ is a cylindrical random variable in $\Phi'$, a $\Phi'$-valued random variable $Y$ is called a \emph{version} of $X$ if for every $\phi \in \Phi$, $X(\phi)=\inner{Y}{\phi}$ $\Prob$-a.e. A $\Phi'$-valued process $Y=(Y_{t}:t \in J)$ is said to be a $\Phi'$-valued \emph{version} of the cylindrical process $X=(X_{t}: t \in J)$ on $\Phi'$ if for each $t \in J$, $Y_{t}$ is a $\Phi'$-valued version of $X_{t}$.  

For a $\Phi'$-valued process $X=( X_{t}: t \in J)$ terms like continuous, c\`{a}dl\`{a}g, purely discontinuous, adapted, predictable, etc. have the usual (obvious) meaning. 

A $\Phi'$-valued random variable $X$ is called \emph{regular} if there exists a weaker countably Hilbertian topology $\theta$ on $\Phi$ (see Section 2 in \cite{FonsecaMora:2018}) such that $\Prob( \omega: X(\omega) \in (\widehat{\Phi}_{\theta})')=1$; here $\Phi_{\theta}$ denotes the space $(\Phi,\theta)$ and $\widehat{\Phi}_{\theta}$ denotes its completion. If $\Phi$ is barrelled, the property of being regular is  equivalent to the property that the law of $Y$ is a Radon measure on $\Phi'$ (see Theorem 2.10 in \cite{FonsecaMora:2018}). A $\Phi'$-valued process $Y=(Y_{t}:t \in J)$ is said to be \emph{regular} if $Y_{t}$ is a regular random variable for each $t \in J$. 

A cylindrical process $Y=(Y_{t}:t \in J)$ in $\Phi'$ is a cylindrical martingale if $Y(\phi)=(Y_{t}(\phi): t \in J)$ is a real-valued martingale for each $\phi \in \Phi$. A $\Phi'$-valued process is a martingale if the induced cylindrical process is a cylindrical martingale in  $\Phi'$.

\subsection{Stochastic Integration in Duals of Nuclear Spaces}\label{subsecStocInteg}

\begin{assumption} From now on $\Phi$ denotes a locally convex space and $\Psi$ denotes a quasi-complete, bornological, nuclear space. 
\end{assumption}

In this section we briefly recall the theory of stochastic integration in duals of nuclear spaces in introduced in  \cite{FonsecaMora:2018-1}. We start with the definition of cylindrical martingale-valued measure.

Let $U$ be a topological space and consider a ring $\mathcal{R}\subseteq \mathcal{B}(U)$ that generates $\mathcal{B}(U)$.  A \emph{cylindrical martingale-valued measure} on $\R_{+} \times \mathcal{R}$ is a collection $M=(M(t,A): t \geq 0, A \in \mathcal{R})$ of cylindrical random variables in $\Phi'$ such that:
\begin{enumerate}
\item $\forall \, A \in \mathcal{R}$, $M(0,A)(\phi)= 0$ $\Prob$-a.s., $\forall \phi \in \Phi$.
\item $\forall t \geq 0$, $M(t,\emptyset)(\phi)= 0$ $\Prob$-a.s. $\forall \phi \in \Phi$ and if $A, B \in \mathcal{R}$ are disjoint then 
$$M(t,A \cup B)(\phi)= M(t,A)(\phi) + M(t,B)(\phi) \, \Prob \mbox{-a.s.}, \quad \forall \phi \in \Phi.$$
\item $\forall \, A \in \mathcal{R}$, $(M(t,A): t \geq 0)$ is a cylindrical zero-mean square integrable c\`{a}dl\`{a}g martingale. 
\item For disjoint $A, B \in \mathcal{R}$, $\Exp \left( M(t,A)(\phi) M(s,B)(\varphi) \right)=0$, for each $t,s \geq 0$, $\phi, \varphi \in \Phi$. 
\end{enumerate}
We will further assume that the following properties are satisfied:
\begin{enumerate}  \setcounter{enumi}{4}
\item For $0\leq s < t$, $M ( (s, t], A)(\phi) \defeq (M(t,A)- M(s,A))(\phi)$ is independent of $\mathcal{F}_{s}$, for all $A \in \mathcal{R}$, $\phi \in \Phi$.
\item  For each $A \in \mathcal{R}$ and $0 \leq s < t$, 
\begin{equation} \label{covarianceFunctionalNuclearMartValuedMeasure}
\Exp \left( \abs{ M((s,t],A)(\phi)}^{2} \right) = \int_{s}^{t} \int_{A} q_{r,u}(\phi)^{2} \mu(du) \lambda (dr) , \quad \forall \, \phi \in \Phi,
\end{equation} 
where 
\begin{enumerate}  
\item $\mu$ is a $\sigma$-finite measure on $(U, \mathcal{B}(U))$ satisfying $\mu(A)< \infty$, $\forall \, A \in \mathcal{R}$,
\item $\lambda$ is a $\sigma$-finite measure on $(\R_{+},\mathcal{B}(\R_{+}))$, finite on bounded intervals,
\item $\{q_{r,u}: r \in \R_{+}, \, u \in U \}$ is a family of  continuous Hilbertian semi-norms on $\Phi$, such that for each $\phi$, $\varphi$ in $\Phi$, the map $(r,u) \mapsto q_{r,u}(\phi,\varphi)$ is $\mathcal{B}(\R_{+}) \otimes \mathcal{B}(U)/ \mathcal{B}(\R_{+})$-measurable and  bounded on $[0,T] \times U$ for all $T>0$. Here, $q_{r,u}(\cdot,\cdot)$ denotes the positive, symmetric, bilinear form associated to the Hilbertian semi-norm $q_{r,u}$.  
\end{enumerate}
\end{enumerate}

Now we recall the main class of integrands introduced in \cite{FonsecaMora:2018-1}.  Let $T>0$. We denote by $\Lambda^{2}_{s}(T)$  the collection of all 
families $R=\{R(r,\omega,u):  r \in [0,T], \omega \in \Omega, u \in U  \} $ of linear operators $R(r,\omega,u) \in \mathcal{L}(\Phi'_{q_{r,u}},\Psi')$  satisfying that  the  mapping  $(r,\omega,u) \mapsto q_{r,u}(R(r,\omega,u)' \psi, \phi)$ is $\mathcal{P}_{T} \otimes \mathcal{B}(U)$-measurable (we will say that the family is \emph{$q_{r,u}$-predictable}) and for which 
\begin{equation*} \label{finiteSecondMomentIntegrands}
\Exp \int_{0}^{T} \int_{U} q_{r,u}(R(r,u)'\psi)^{2} \mu(du) \lambda(dr) < \infty, \quad \forall \, \psi \in \Psi.
\end{equation*}

For each $R \in \Lambda^{2}_{s}(T)$, there exist a unique (up to indistinguishability) $\Phi'$-valued, regular, square integrable martingale with c\`{a}dl\`{a}g paths, $ \int_{0}^{t} \int_{U} R(r,u) M(dr,du)$, $t \in [0,T]$, which we call the \emph{strong stochastic integral} of $R$. Details of the construction can be found in Section 5 in \cite{FonsecaMora:2018-1}. 

There exists a particular subclass of integrands which plays an important role on the study of properties of the stochastic integral and which consists of operator-valued mappings with range on a Hilbert space embedded in $\Psi'$. 

Let $p$ be a continuous Hilbertian semi-norm on $\Psi$. Let $\Lambda^{2}_{s}(p,T)$ denote the collection of families $\tilde{R}=\{\tilde{R}(r,\omega,u)\}$ of linear operators $\tilde{R}(r,\omega,u) \in \mathcal{L}_{2}(\Phi'_{q_{r,u}},\Psi'_{p})$, $r \in [0, T]$, $\omega \in \Omega$, $u \in U$, which are $q_{r,u}$-predictable, and for which 
\begin{equation} \label{finiteSecondMomentIntegrandsStrongIntgInHilbertSpace*}
\Exp \int_{0}^{T} \int_{U} \norm{\tilde{R}(r,u)}^{2}_{\mathcal{L}_{2}(\Phi'_{q_{r,u}},\Psi'_{p})} \mu(du) \lambda(dr) < \infty.
\end{equation} 
For $\tilde{R} \in \Lambda^{2}_{s}(p,T)$, it follows by Theorem 3.3.16 in \cite{FonsecaMoraThesis} that the stochastic integral 
$ \int_{0}^{t} \int_{U} \tilde{R}(r,\omega,u) M(dr,du)$ is a $\Phi'_{p}$-valued zero-mean square integrable martingale with c\`{a}dl\`{a}g paths satisfying for each $0 \leq t \leq T$, 
\begin{equation}\label{eqItoIsometryStrongInte}
\Exp \left( p'\left( \int_{0}^{t} \int_{U} \tilde{R}(r,u) M(dr,du) \right)^{2} \right)= \Exp \int_{0}^{t} \int_{U} \norm{\tilde{R}(r,u)}^{2}_{\mathcal{L}_{2}(\Phi'_{q_{r,u}},\Psi'_{p})} \mu(du) \lambda(dr)
\end{equation}
Moreover, by Theorem 5.11 in \cite{FonsecaMora:2018-1} for every $R \in \Lambda^{2}_{s}(T)$ there exists a continuous Hilbertian seminorm $p$ on $\Phi$ and some $\tilde{R} \in \Lambda^{2}_{s}(p,T)$ such that $R(r,\omega,u)=i'_{p} \tilde{R}(r,\omega,u)$ for $\lambda \otimes \Prob \otimes \mu$-a.e. $(r,\omega,u)$ and the stochastic integral 
$ \int_{0}^{t} \int_{U} \tilde{R}(r,\omega,u) M(dr,du)$ is a $\Phi'_{p}$-valued version of $ \int_{0}^{t} \int_{U} R(r,\omega,u) M(dr,du)$. 

\section{Regularity of Paths of the Stochastic convolution}\label{sectRegulaPathsStochConvo}

\begin{assumption}
Al through this section $M=(M(t,A): t \geq 0, A \in \mathcal{R})$ denotes a cylindrical martingale-valued measure as in Section \ref{subsecStocInteg} with $\lambda$ being the Lebesgue measure on $(\R_{+}, \mathcal{B}(\R_{+}))$, and let $R \in \Lambda^{2}_{s}(T)$. 
\end{assumption}

Let $(S(t): t \geq 0)$ be a $(C_{0},1)$-semigroup on $\Psi$.  
In Proposition 6.8 in \cite{FonsecaMora:2018-1} it is shown that the stochastic convolution 
\begin{equation}\label{defiStochCOnvolu}
X_{t}=\int_{0}^{t} \int_{U} S(t-r)' R(r,u)  M(dr,du), \quad \forall \, t \in [0,T],
\end{equation}
is a $\Psi'$-valued regular, adapted, square integrable process. Moreover, Theorem 6.14 in \cite{FonsecaMora:2018-1} shows that there exists a continuous Hilbertian seminorm $q$ on $\Phi$ such that the stochastic convolution has a $\Phi'_{q}$-valued, mean-square continuous, predictable version with second moments. In the following result we show that  if $(S(t): t \geq 0)$ is Hilbertian, then a  c\`{a}dl\`{a}g version exists. 
 
\begin{theorem} \label{theoCadlagVerStochConv} 
Assume that $(S(t): t \geq 0)$ is a Hilbertian $(C_{0},1)$-semigroup on $\Psi$. Then there exists a continuous Hilbertian seminorm $p$ on $\Phi$ such that the stochastic convolution $(X_{t}: t \in [0,T])$ given in \eqref{defiStochCOnvolu} has a $\Psi'_{p}$-valued square integrable adapted c\`{a}dl\`{a}g version $(Y_{t}: t \in [0,T])$ satisfying $\sup_{t \in [0,T]} \Exp( p'(Y_{t})^{2})  < \infty$. In particular, the stochastic convolution has a $\Psi'$-valued c\`{a}dl\`{a}g version.  

Moreover if for each $A \in \mathcal{R}$ and $\phi \in \Phi$, the real-valued process $(M(t,A)(\phi): t \geq 0)$ is continuous, then the results above remain valid replacing the property c\`{a}dl\`{a}g by continuous.
\end{theorem}

For our proof of  Theorem \ref{theoCadlagVerStochConv} we will require the following terminology. For $(S(t): t \geq 0)$ let $\Pi$ be the corresponding family of Hilbertian seminorms satisfying \eqref{eqDefC01Semigroup}. 
For each $p \in \Pi$ it is shown in (\cite{Babalola:1974}, Theorems 2.3 and 2.6) that there exists a $C_{0}$-semigroup $( S_{p}(t): t \geq 0)$  on the Banach space (Hilbert space since $p$ is Hilbertian) $\Psi_{p}$ into itself such that
\begin{equation} \label{eqDefExtenSemiGroupToBanach}
 S_{p}(t) i_{p} \psi= i_{p} S(t) \psi, \quad \forall \, \psi \in \Psi, \, t \geq 0.
\end{equation}
Our first step is to prove the following Kotelenez type inequality.

\begin{lemma} \label{lemmKotelenezIneq} Let $(S(t): t \geq 0)$ be a Hilbertian $(C_{0},1)$-semigroup with corresponding family $\Pi$ of Hilbertian seminorms satisfying \eqref{eqDefC01Semigroup}.  For any continuous Hilbertian seminorm $p \in \Pi$, $F \in \Lambda^{2}_{s}(p,T)$, $C>0$ and countable $D \subseteq [0,T]$, 
\begin{multline*}
 \Prob \left( \sup_{t \in D} p' \left( \int^{t}_{0} \int_{U} S_{p}(t-r)'F(r,u) \, M(dr,du) \right)>C  \right) \\
\leq \frac{e^{2\theta_{p} T}}{C^{2}} \Exp \int_{0}^{T} \int_{U} \norm{F(r,u)}^{2}_{\mathcal{L}_{2}(\Phi'_{q_{r,u}},\Psi'_{p})}  \mu(du)dr,
\end{multline*}
where $( S_{p}(t): t \geq 0)$ is the corresponding $C_{0}$-semigroup on the Hilbert space $\Psi_{p}$ satisfying \eqref{eqDefC01Semigroup} (replacing $S(t)$ with $S_{p}(t)$) and \eqref{eqDefExtenSemiGroupToBanach}.
\end{lemma}
\begin{proof}
We modify to our context the arguments used in the proof of Proposition 9.18 in \cite{PeszatZabczyk}. 
  
For fixed $t \in [0,T]$, one can easily check that the family $\{ S_{p}(t-r)' F(r,\omega,u): r \in[0,t], \omega \in \Omega, u \in U \}$  belongs to $\Lambda^{2}_{s}(p,t)$. In particular, 
\begin{multline*}
\Exp \int_{0}^{t} \int_{U} \norm{S_{p}(t-r)' F(r,u)}^{2}_{\mathcal{L}_{2}(\Phi'_{q_{r,u}},\Psi'_{p})} \mu(du) dr \\ \leq e^{2\theta_{p}t} \Exp \int_{0}^{t} \int_{U} \norm{F(r,u)}^{2}_{\mathcal{L}_{2}(\Phi'_{q_{r,u}},\Psi'_{p})} \mu(du) dr < \infty.
\end{multline*}

Hence, from Theorem 3.3.16 in \cite{FonsecaMoraThesis} the stochastic convolution 
$$Y(t) \defeq \int^{t}_{0} \int_{U} \, S_{p}(t-r)' F(r,u) \, M(dr,du), \quad t \in [0,T], $$
is a $\Phi'_{p}$-valued adapted process such that for each $t \in [0,T]$,
\begin{multline}\label{eqItoIsometrySpacePhiq}
\Exp \,\left[ p'\left( \int^{t}_{0} \int_{U} S_{p}(t-r)' F(r,u) \, M(dr,du) \right)^{2} \right]\\
=  \int^{t}_{0} \int_{U} \norm{S_{p}(t-r)' F(r,u)}^{2}_{\mathcal{L}_{2}(\Phi'_{q_{r,u}},\Psi'_{p})}  \mu(du)dr < \infty.  
\end{multline}

Let $0=t_{0}<t_{1}< \dots < t_{n}=T$ and $C>0$. Then, 
\begin{eqnarray}
\Prob \left( \max_{1 \leq k \leq n} p'(Y(t_{k}))  >C\right) & = & \sum_{k=1}^{n} \Prob \left( \cap_{j=1}^{k-1} \{ p'(Y(t_{j}))\leq C\} \cap \{ p'(Y(t_{k}))>C\} \right) \nonumber \\
& \leq & \frac{1}{C^{2}} \sum_{k=1}^{n} \Exp ( p'(Y(t_{k}))^{2} \mathbbm{1}_{\{ p'(Y(t_{k}))>C \}}    \mathbbm{1}_{k} ), \label{eqMaxInequForKotelenez}
\end{eqnarray}
where $\mathbbm{1}_{k}= \prod_{j=1}^{k-1} \mathbbm{1}_{\{ p'(Y(t_{j}) \leq C \}}$. Now, observe that for each $1 \leq k \leq n$, the semigroup property of $(S_{p}(t): t \geq 0)$ and the action of the continuous linear operators on the stochastic integral (as for example in Proposition 5.18 in \cite{FonsecaMora:2018-1}) show that we have
$$Y(t_{k})= S_{p}(t_{k}-t_{k-1})' \,Y(t_{k-1})+ \int_{t_{k-1}}^{t_{k}} \int_{U}  S_{p}(t_{k}-r)'F(r,u)  M(dr,du).$$
Then, by the martingale property of the stochastic integral, the It\^{o} isometry \eqref{eqItoIsometrySpacePhiq} and  \eqref{eqDefC01Semigroup} (replacing $S(t)$ with $S_{p}(t)$), we get 
\begin{flalign}
 & \Exp ( p'(Y(t_{k}))^{2} \mathbbm{1}_{k} ) \nonumber \\
 & = \Exp ( p'(S_{p}(t_{k}-t_{k-1})' \,Y(t_{k-1}))^{2} \mathbbm{1}_{k} )+\Exp \left( p' \left( \int_{t_{k-1}}^{t_{k}} \int_{U}  S_{p}(t_{k}-r)'F(r,u)M(dr,du) \right)^{2} \mathbbm{1}_{k} \right) \nonumber \\
 & \leq e^{2 \theta_{p}(t_{k}-t_{k-1})} \left( \Exp ( p'(Y(t_{k-1}))^{2} \mathbbm{1}_{k-1} )+ \Exp \int_{t_{k-1}}^{t_{k}} \int_{U}  \norm{F(r,u)}^{2}_{\mathcal{L}_{2}(\Phi'_{q_{r,u}},\Psi'_{p})} \mu(du) dr \right).  \nonumber
 \end{flalign}
Then, by iteration we have
\begin{eqnarray*}
\sum_{k=1}^{n} \Exp ( p'(Y(t_{k}))^{2} \mathbbm{1}_{k} ) 
& \leq & e^{2 \theta_{p}T} \sum_{k=1}^{n} \Exp \int_{t_{k-1}}^{t_{k}} \int_{U}  \norm{F(r,u)}^{2}_{\mathcal{L}_{2}(\Phi'_{q_{r,u}},\Psi'_{p})} \mu(du) dr \\
& = & e^{2 \theta_{p}T} \Exp \int_{0}^{T} \int_{U}  \norm{F(r,u)}^{2}_{\mathcal{L}_{2}(\Phi'_{q_{r,u}},\Psi'_{p})} \mu(du) dr.
\end{eqnarray*}
Combining the above inequality with \eqref{eqMaxInequForKotelenez} we conclude the estimate in Lemma \ref{lemmKotelenezIneq}. 
\end{proof}

\begin{proof}[Proof of Theorem \ref{theoCadlagVerStochConv}] 
Let $p$ be a continuous Hilbertian seminorm on $\Phi$ and $\tilde{R} \in \Lambda^{2}_{s}(p,T)$ such that $R(r,\omega,u)=i'_{p} \tilde{R}(r,\omega,u)$ $\lambda \otimes \Prob \otimes \mu$-a.e.
Since the family of Hilbertian seminorms $\Pi$ generates the topology on $\Psi$, then we can assume $p \in \Pi$. The corresponding $C_{0}$-semigroup $( S_{p}(t): t \geq 0)$ on the Hilbert space $\Psi_{p}$ satisfies \eqref{eqDefC01Semigroup} (replacing $S(t)$ with $S_{p}(t)$) and \eqref{eqDefExtenSemiGroupToBanach}. 

Then, for fixed $t \in [0,T]$ it follows from the above properties that for $\mbox{Leb}\otimes\mu-\mbox{a.e. } (r,u) \in [0,T]\times U$, 
\begin{equation}\label{eqDecomIntegrandStocConv}
\ind{[0,t]}{r} S(t-r)'R(r,\omega,u)=i'_{p} S_{p}(t-r)' \tilde{R}(r,\omega,u).  
\end{equation} 
Moreover, the family $\{ S_{p}(t-r)' \tilde{R}(r,\omega,u): r \in[0,t], \omega \in \Omega, u \in U \}$  belongs to $\Lambda^{2}_{s}(p,t)$ and the  stochastic convolution 
$$\int^{t}_{0} \int_{U} \, S_{p}(t-r)' \tilde{R}(r,u) \, M(dr,du), \quad t \in [0,T], $$
is a $\Psi'_{p}$-valued adapted process satisfying the It\^{o} isometry \eqref{eqItoIsometrySpacePhiq} (with $F$ being replaced by $ \tilde{R}$). Furthermore, it follows from \eqref{eqDecomIntegrandStocConv}, Theorem 3.3.17 in \cite{FonsecaMoraThesis}, and Theorem 5.11 in \cite{FonsecaMora:2018-1}, that for each $t \in [0,T]$, $\Prob$-a.e.
\begin{equation}\label{eqEqualityASStochConvo}
\int_{0}^{t} \int_{U} S(t-r)' R(r,u)  M(dr,du) =
i'_{p} \int^{t}_{0} \int_{U} S_{p}(t-r)' \tilde{R}(r,u) \,   M(dr,du). 
\end{equation}
Therefore, $\displaystyle{\int^{t}_{0} \int_{U} S_{p}(t-r)' \tilde{R}(r,u)  \,  M(dr,du)}$ is a $\Psi'_{p}$-valued square integrable version for the  stochastic convolution $\displaystyle{\int_{0}^{t} \int_{U} S(t-r)' R(r,u) \,  M(dr,du)}$. 

Hence in order to prove Theorem \ref{theoCadlagVerStochConv} it suffices to show  that the stochastic convolution $\displaystyle{\int^{t}_{0} \int_{U} S_{p}(t-r)' \tilde{R}(r,u) \,   M(dr,du)}$ has a c\`{a}dl\`{a}g version.

Given $k \in \N$, let $r(k)=iT/2^{k}$ if $r \in ( iT/2^{k}, (i+1)T/2^{k}]$ for $i=0,1, \dots, 2^{k-1}$ and consider the $\Psi'_{p}$-valued process
$$Y^{k}(t) = \int^{t}_{0} \int_{U} S_{p}(t-r(k))' \tilde{R}(r,u) \,   M(dr,du), \quad \forall \, t \in [0,T].$$
Observe that because for each $t \in [0,T]$, 
\begin{multline*}
\Exp \int^{t}_{0} \int_{U} \norm{(S_{p}(t-r(k))'-S_{p}(t-r)') \tilde{R}(r,u)}^{2}_{\mathcal{L}_{2}(\Phi'_{q_{r,u}},\Psi'_{p})}  \mu(du)dr \\
\leq 2e^{2\theta_{p}t} \Exp \int_{0}^{T} \int_{U} \norm{\tilde{R}(r,u)}^{2}_{\mathcal{L}_{2}(\Phi'_{q_{r,u}},\Psi'_{p})} \mu(du) \lambda(dr) < \infty,
\end{multline*}
then by the strong continuity of the $C_0$-semigroup $S_{p}(t)$ and dominated convergence we have that 
$$ \lim_{k \rightarrow \infty} \Exp \int^{t}_{0} \int_{U} \norm{(S_{p}(t-r(k))'-S_{p}(t-r)') \tilde{R}(r,u)}^{2}_{\mathcal{L}_{2}(\Phi'_{q_{r,u}},\Psi'_{p})}  \mu(du)dr=0.$$
Therefore, by the It\^{o} isometry \eqref{eqItoIsometrySpacePhiq} (with $F$ being replaced by $ \tilde{R}$) we conclude  that
$$ \lim_{k \rightarrow \infty} \Exp \left[ \, p' \left( \int^{t}_{0} \int_{U} (S_{p}(t-r(k))'-S_{p}(t-r)') \tilde{R}(r,u) \,  M(dr,du) \right)^{2} \right] =0.$$
So for each $t \in [0,T]$, $Y^{k}(t)$ converges in  $L^{2} (\Omega, \mathcal{F}, \Prob; \Psi'_{p})$ (the space of square integrable random variables in $\Psi'_{p}$) to $\int^{t}_{0} \int_{U} S_{p}(t-r)' \tilde{R}(r,u) \,   M(dr,du)$. 

Now, observe that for each $k \in \N$, the process $Y^{k}$ is c\`{a}dl\`{a}g. To see why this is true, note that if $t \in (iT/2^{k}, (i+1)T/2^{k}]$, then 
\begin{multline*}
Y^{k}(t)=S_{p}\left(t-\frac{iT}{2^{k}} \right)' \int^{t}_{\frac{iT}{2^{k}}} \int_{U} \tilde{R}(r,u) \,  M(dr,du) \\
+\sum_{j=1}^{i} S_{p} \left( t-\frac{(j-1)T}{2^{k}} \right)'  \int^{\frac{jT}{2^{k}}}_{\frac{(j-1)T}{2^{k}}} \int_{U} \tilde{R}(r,u) \,  M(dr,du).
\end{multline*}
Thus, the strong continuity of the semigroup $S_{p}$ and because the stochastic integral $\int_{0}^{t} \int_{U} \widetilde{R}(r,u) \,  M(dr,du)$ have c\`{a}dl\`{a}g paths we conclude that  each $Y^{k}$ is c\`{a}dl\`{a}g. 

Our next objective is to show that the sequence $(Y^{k}: k \in \N)$ has a subsequence that converges in probability uniformly on $[0,T]$. In that case, it will follows that  
$$ \int^{t}_{0} \int_{U} S_{p}(t-r)'\tilde{R}(r,u) \,  M(dr,du), \quad \forall \, t \in [0,T],$$
has a c\`{a}dl\`{a}g version. 

Assume $m \geq k$. Then, because $r(m) \geq r(k)$, we have for every $t \in [0,T]$ that  
\begin{flalign}
& Y^{m}(t)-Y^{k}(t) \nonumber \\
& =  \int^{t}_{0} \int_{U} (S_{p}(t-r(m))'-S_{p}(t-r(k))') \tilde{R}(r,u) \, M(dr,du) \nonumber \\
& =  \int^{t}_{0} \int_{U} S_{p}(t-r(m))'(I'_{p}-S_{p}(r(m)-r(k))')  \tilde{R}(r,u) \,M(dr,du), \label{eqDiferenceApproxOU} 
\end{flalign}
where $I_{p}$ is the identity in $\Psi_{p}$. Let $F^{m,k}(r,\omega,u)=(I'_{p}-S_{p}(r(m)-r(k))') \tilde{R}(r,\omega,u)$ for $(r,\omega,u) \in [0,T] \times \Omega \times U$. By \eqref{eqDefC01Semigroup} (replacing $S(t)$ with $S_{p}(t)$) we have  
\begin{multline}\label{eqNormIntegraFmk}
\Exp \int_{0}^{T} \int_{U} \norm{F^{m,k}(r,u)}^{2}_{\mathcal{L}_{2}(\Phi'_{q_{r,u}},\Psi'_{p})} \mu(du) dr \\ \leq (1+e^{2\theta_{p}T}) \Exp \int_{0}^{T} \int_{U} \norm{\tilde{R}(r,u)}^{2}_{\mathcal{L}_{2}(\Phi'_{q_{r,u}},\Psi'_{p})} \mu(du) dr < \infty.
\end{multline}
Hence $F^{m,k} \in \Lambda^{2}_{s}(p,T)$. Then, from Lemma \ref{lemmKotelenezIneq}, the fact that $Y^{m}$ and $Y^{k}$ are c\`{a}dl\`{a}g and from \eqref{eqDiferenceApproxOU} it follows that
\begin{multline}
 \Prob \left( \sup_{t \in [0,T]} p'\left( Y^{m}(t)-Y^{k}(t)  \right)^{2}>C  \right) \\ 
\leq \frac{e^{2\theta_{p}T}}{C^{2}} \Exp \int_{0}^{T} \int_{U} \norm{F^{m,k}(r,u)}^{2}_{\mathcal{L}_{2}(\Phi'_{q_{r,u}},\Psi'_{p})}  \mu(du)dr. \label{estimateDiferApproxOU}
\end{multline}
Now, the strong continuity of the $C_{0}$-semigroup $S_{p}(t)$ shows that for each $(r,\omega,u) \in [0,T] \times \Omega \times U$ we have
$$ \lim_{m,k \rightarrow \infty} \norm{F^{m,k}(r,\omega,u)}^{2}_{\mathcal{L}_{2}(\Phi'_{q_{r,u}},\Psi'_{p})} =0.$$
Therefore, from  \eqref{eqNormIntegraFmk}, dominated convergence and \eqref{estimateDiferApproxOU} it follows that
$$\lim_{m,k \rightarrow \infty} \Prob \left( \sup_{t \in [0,T]} p'\left( Y^{m}(t)-Y^{k}(t)  \right)^{2}>C  \right)=0.$$ 
Then, by standard arguments we can show the existence of a subsequence of $(Y^{k}: k \in \N)$ that converges $\Prob$-a.e. uniformly on $[0,T]$ to a c\`{a}dl\`{a}g version $(Y_{t}: t \in [0,T])$ of the process $\int^{t}_{0} \int_{U} S_{p}(t-r)' \tilde{R}(r,u) \,  M(dr,du)$, $t \in [0,T]$. Moreover, observe that by \eqref{eqItoIsometrySpacePhiq} we have
$$ \sup_{t \in [0,T]} \Exp( p'(Y_{t})^{2}) \leq e^{2\theta_{p} T} \Exp \int_{0}^{T} \int_{U} \norm{\tilde{R}(r,u)}^{2}_{\mathcal{L}_{2}(\Phi'_{q_{r,u}},\Psi'_{p})} \mu(du) dr < \infty. $$

Finally, if each $(M(t,A)(\phi): t \geq 0)$ is continuous
then the stochastic integrals defined with respect to $M$ are continuous processes (see Proposition 5.12 in \cite{FonsecaMora:2018-1}). In such a case each member of the approximation sequence $(Y^{k}: k \in \N)$ is continuous and since there exist a subsequence that converges $\Prob$-a.e. uniformly on $[0,T]$ we conclude that $\int^{t}_{0} \int_{U} S_{p}(t-r)' \tilde{R}(r,u) \,  M(dr,du)$ has a continuous version. This completes the proof of Theorem \ref{theoCadlagVerStochConv}. 
\end{proof}

\section{Time regularity for solutions to stochastic evolution equations}\label{sectTimeReguSEE}

In this section we apply our result in Theorem \ref{theoCadlagVerStochConv} to show the existence of c\`{a}dl\`{a}g solutions to the following class of stochastic evolution equations
\begin{equation}\label{generalFormSEE}
d X_{t}= (A'X_{t}+ B (t,X_{t})) dt+\int_{U} F(t,u,X_{t}) M (dt,du), \quad \mbox{for }t \geq 0,
\end{equation}
with initial condition $X_{0}=Z_{0}$, where we will assume the following: 

\begin{assumption} \label{assumptionsCoefficients} \hfill 

\noindent \textbf{(A1)} Every continuous seminorm on $\Psi'$ is separable. 

\noindent \textbf{(A2)} $Z_{0}$ is a $\Psi'$-valued, regular, $\mathcal{F}_{0}$-measurable, square-integrable random variable. 

\noindent \textbf{(A3)}  $A$ is the infinitesimal generator of a $(C_{0},1)$-semigroup $(S(t): t\geq 0)$ on $\Psi$ and the dual semigroup $(S(t)': t \geq 0)$ is a $(C_{0},1)$-semigroup. 

\noindent \textbf{(A4)} $M=(M(t,A): t \geq 0, A \in \mathcal{R})$ is a cylindrical martingale-valued measure as in Section \ref{subsecStocInteg} with $\lambda$ being the Lebesgue measure on $(\R_{+}, \mathcal{B}(\R_{+}))$.

\noindent \textbf{(A5)} $B:\R_{+} \times \Psi' \rightarrow \Psi'$ is such that the map $(r,g) \mapsto \inner{B(r,g)}{\psi}$ is $ \mathcal{B}(\R_{+}) \otimes \mathcal{B} (\Psi')$-measurable, for every $\psi \in \Psi$.

\noindent \textbf{(A6)} $F= \{F(r,u,g): r \in \R_{+}, u \in U, g \in \Psi'\}$ is such that
\begin{enumerate}[label=(\alph*)]
\item $ F(r,u,g) \in \mathcal{L} (\Phi'_{q_{r,u}}, \Psi')$, $\forall r \in \R_{+}$, $u \in U$, $g \in \Psi'$. 
\item The mapping $(r,u,g)\mapsto q_{r,u}(F(r,u,g)'\psi,\phi)$ is $ \mathcal{B}(\R_{+}) \otimes \mathcal{B}(U) \otimes \mathcal{B}(\Psi')$-measurable, for every $\phi \in \Phi$, $\psi \in \Psi$.
\end{enumerate}

\noindent \textbf{(A7)}
There exist two functions $a, b: \Psi \times \R_{+} \rightarrow \R_{+}$ satisfying:
\begin{enumerate}
\item For each $T>0$ and $K \subseteq \Psi$ bounded, 
\begin{equation*}
 \int_{0}^{T} \sup_{\psi \in K} ( a(\psi,r)^{2} +  b(\psi,r)^{2}) dr < \infty. 
\end{equation*}
\item (Growth conditions) For all $r \in \R_{+}$, $g \in \Psi'$,   
\begin{align*}
\abs{\inner{B(r,g)}{\psi}} & \leq  a(\psi,r)(1+\abs{\inner{g}{\psi}}), \\
\int_{U} q_{r,u}(F(r,u,g)'\psi)^{2} \mu(du) & \leq b(\psi,r)^{2}(1+\abs{\inner{g}{\psi}})^{2}. 
\end{align*} 
\item (Lipschitz conditions) For all $r \in \R_{+}$, $g_{1}, g_{2} \in \Psi'$,  
\begin{align*}
\abs{\inner{B(r,g_{1})}{\psi}-\inner{B(r,g_{2})}{\psi}} & \leq  a(\psi,r) \abs{\inner{g_{1}}{\psi}-\inner{g_{2}}{\psi}}, \\
\int_{U} q_{r,u}(F(r,u,g_{1})'\psi-F(r,u,g_{2})'\psi)^{2} \mu(du)  & \leq b(\psi,r)\abs{\inner{g_{1}}{\psi}-\inner{g_{2}}{\psi}}^{2}. 
\end{align*}
\end{enumerate}
\end{assumption}

It is shown in Theorem 6.23 in \cite{FonsecaMora:2018-1} that \eqref{generalFormSEE} has a unique \emph{weak solution} with initial condition $X_{0}=Z_{0}$, that is a  $\Psi'$-valued regular and predictable process $X=(X_{t}: t \geq 0)$ such that for every $\psi \in \mbox{Dom}(A)$ and $ t \geq 0$, $\Prob$-a.e.
$$
\inner{X_{t}}{\psi} = \inner{X_{0}}{\psi}+ \int^{t}_{0}\inner{X_{r}}{A \psi}+ \inner{B (r,X_{r})}{\psi}dr  +
\int^{t}_{0} \int_{U} F(r,u,X_{r})'\psi M(dr,du),
$$
where the first integral is a Lebesgue integral for $\Prob$-a.e. $\omega \in \Omega$ and the second integral is a weak stochastic integral as defined in Section 4 in \cite{FonsecaMora:2018-1}. 

This solution is also a \emph{mild solution}, i.e. for every $ t \geq 0$, $\Prob$-a.e.
\begin{equation} \label{equationMildSolution}
X_{t} = S(t)'Z_{0}+ \int^{t}_{0} S(t-r)' B(r,X_{r})dr  +  \int^{t}_{0} \int_{U} S(t-r)' F(r,u,X_{r}) M(dr,du).
\end{equation}
The first integral at the right-hand side of \eqref{equationMildSolution} is a $\Psi'$-valued regular, adapted process such that for all $t \geq 0$ and $\psi \in \Psi$, for $\Prob$-a.e. $\omega \in \Omega$,
\begin{equation} \label{defNonRandomConvolutionIntegral}
\inner{ \int^{t}_{0} S(t-r)'B(r,X_{r}(\omega))dr }{\psi}= \int^{t}_{0} \inner{S(t-r)' B(r,X_{r}(\omega))}{\psi} dr,
\end{equation}
where for each $t \geq 0$, $\psi \in \Psi$, the integral on the right-hand side of \eqref{defNonRandomConvolutionIntegral} is defined for $\Prob$-a.e. $\omega \in \Omega$ as a Lebesgue integral. The second integral at the right-hand side of \eqref{equationMildSolution} is the stochastic convolution. 

For every $T>0$, it is proved in  Theorem 6.23 in \cite{FonsecaMora:2018-1} that there exists a continuous Hilbertian seminorm $\rho=\rho(T)$ on $\Psi$ such that $X=( X_{t} : t \in [0,T])$ has a $\Psi'_{\rho}$-valued predictable version $\tilde{X}=( \tilde{X}_{t} : t \in [0,T])$ satisfying $\sup_{t \in [0,T]} \Exp \left( \rho'(\tilde{X}_{t})^{2} \right)< \infty$. 

\begin{theorem}\label{theoCadlagSolSEEwithCMVM}
Apart from Assumption \ref{assumptionsCoefficients}, suppose we have that $(S(t): t \geq 0)$ is a Hilbertian $(C_{0},1)$-semigroup on $\Psi$. Then \eqref{generalFormSEE} has a unique c\`{a}dl\`{a}g weak solution with initial condition $X_{0}=Z_{0}$. Moreover, for every $T>0$ there exists a continuous Hilbertian seminorm $\rho=\rho(T)$ on $\Psi$ such that $X=( X_{t} : t \in [0,T])$ is a $\Psi'_{\rho}$-valued adapted c\`{a}dl\`{a}g process satisfying $\sup_{t \in [0,T]} \Exp \left( \rho'(X_{t})^{2} \right)< \infty$. 

Moreover if for each $A \in \mathcal{R}$ and $\phi \in \Phi$, the real-valued process $(M(t,A)(\phi): t \geq 0)$ is continuous, then the results above remain valid replacing the property c\`{a}dl\`{a}g by continuous.
\end{theorem}
\begin{proof}
We already know that \eqref{generalFormSEE} has a unique weak  solution with initial condition $X_{0}=Z_{0}$, hence to prove the result it suffices to show that for every $T>0$ each term in \eqref{equationMildSolution} has a Hilbert space-valued version satisfying the conditions in Theorem \ref{theoCadlagSolSEEwithCMVM}. 

In effect, in the proof of Lemma 6.24 in \cite{FonsecaMora:2018-1} (Step 1) it is shown the existence of a continuous Hilbertian seminorm $\varrho_{0}$ on $\Psi$ such that $(Y^{0}_{t}: t \in [0,T]) \defeq ( S(t)'Z_{0} : t \in [0,T])$ has a $\Psi'_{\varrho_{0}}$-valued continuous adapted version $(\widetilde{Y}^{0}_{t}: t \in [0,T])$ satisfying $\sup_{t \in [0,T]} \Exp \left( \varrho_{0}' ( \widetilde{Y}^{0}_{t})^{2} \right)< \infty$.

Likewise, for $(Y^{1}_{t}: t \in [0,T]) \defeq \left(  \int^{t}_{0} S(t-r)' B(r,X_{r})dr : t \in [0,T] \right)$ it is shown in  the proof of Lemma 6.24 in \cite{FonsecaMora:2018-1} (Step 2) that there exists a continuous Hilbertian seminorm $\varrho_{1}$ on $\Psi$ such that $(Y^{1}_{t}: t \in [0,T])$ has a $\Psi'_{\varrho_{1}}$-valued continuous adapted version $(\widetilde{Y}^{1}_{t}: t \in [0,T])$ satisfying $\sup_{t \in [0,T]} \Exp \left( \varrho_{1}' ( \widetilde{Y}^{1}_{t})^{2} \right)< \infty$.

Finally, for $(Y^{2}_{t}: t \in [0,T]) \defeq \left(  \int^{t}_{0} \int_{U} S(t-r)' F(r,u,X_{r}) M(dr,du) : t \in [0,T] \right)$ it is shown in Lemma 6.24 in \cite{FonsecaMora:2018-1} (Step 3) that $F_{X}= \{ F(r,u,X_{r}(\omega)):  r \in [0,T], \omega \in \Omega, u \in U \} \in \Lambda^{2}_{s}(T)$. Then, Theorem \ref{theoCadlagVerStochConv} shows that there exists a 
continuous Hilbertian seminorm $\varrho_{2}$ on $\Psi$ such that $(Y^{2}_{t}: t \in [0,T])$ has a $\Psi'_{\varrho_{2}}$-valued c\`{a}dl\`{a}g adapted version $(\widetilde{Y}^{2}_{t}: t \in [0,T])$ satisfying $\sup_{t \in [0,T]} \Exp \left( \varrho_{2}' ( \widetilde{Y}^{2}_{t})^{2} \right)< \infty$.

Let $\rho$ be a continuous Hilbertian semi-norm on $\Psi$ such that $\varrho_{i} \leq \rho$, for $i=0,1,2$. Then, the inclusions $i_{\varrho_{i},\rho}:\Psi_{\rho} \rightarrow \Psi_{\varrho_{i}}$, $i=0,1,2$ are linear and continuous. Hence, if we take 
$$Y_{t}= i'_{\varrho_{0},\rho} \widetilde{Y}^{0}_{t} + i'_{\varrho_{1},\rho}  \widetilde{Y}^{1}_{t}  + i'_{\varrho_{2},\rho}  \widetilde{Y}^{2}_{t} ,$$
we have that $(Y_{t}: t \in [0,T])$ is a $\Psi'_{\rho}$-valued adapted c\`{a}dl\`{a}g version for  $X=( X_{t} : t \in [0,T])$ for which it is true that $\sup_{t \in [0,T]} \Exp \left( \rho'(Y_{t})^{2} \right)< \infty$. 

Finally, if for each $A \in \mathcal{R}$ and $\phi \in \Phi$, the real-valued process $(M(t,A)(\phi): t \geq 0)$ is continuous, then Theorem \ref{theoCadlagVerStochConv} shows that $(\widetilde{Y}^{2}_{t}: t \in [0,T])$ has continuous paths, hence the same is true for $(Y_{t}: t \in [0,T])$. 
\end{proof}

\section{Examples and applications}\label{sectAppliExample}

\begin{example} The following is an example of the construction of a Fr\'{e}chet nuclear space and a Hilbertian $(C_{0},1)$-semigroup on it which is commonly used on the literature of stochastic analysis in duals of nuclear spaces. See \cite{KallianpurXiong}, Example 1.3.2, for full details. 

Let $(H, \inner{\cdot}{\cdot}_{H})$ be a separable Hilbert space and $-L$ be a closed densely defined self-adjoint operator on $H$ such that $\inner{-L \phi}{\phi}_{H} \leq 0$ for each $\phi \in \mbox{Dom}(L)$. Let $(S(t): t \geq 0)$ be the $C_{0}$-contraction semigroup on $H$ generated by $-L$. Assume moreover that there exists some $r_{1}$ such that $(\lambda I+ L)^{-r_{1}}$ is Hilbert-Schmidt. Given these conditions, there exist a complete orthonormal system $(\phi_{i}: i \in \N)$ in $H$ and $0 \leq \lambda_{1} \leq \lambda_{2} \leq \cdots$ such that $L \phi_{i} = \lambda_{i} \phi_{i}$ for $i=1,2,\dots$.    

Define 
$$\Psi = \left\{ \psi \in H: \sum_{j=1}^{\infty} (1+\lambda_{j})^{2r} \inner{\psi}{\phi_{j}}^{2}_{H} < \infty, \, \forall r \in \R \right\},$$
and for every $\psi \in \Psi$ and $r \in \R$, define the Hilbertian norm
$$ \abs{\psi}_{r}^{2} = \sum_{j=1}^{\infty} (1+\lambda_{j})^{2r} \inner{\psi}{\phi_{j}}^{2}_{H}. $$
It can be shown that $\Psi$ is a nuclear Fr\'{e}chet space when  equipped with the topology generated by the family  $(\abs{\cdot}_{n}: n \geq 0)$. Moreover,   $(S(t): t \geq 0)$ restricts to a equicontinuous $C_{0}$-semigroup $(S(t): t \geq 0)$ on $\Psi$, i.e. $\abs{S(t)\psi}_{n} \leq \abs{\psi}_{n}$, $n \geq 0$. Hence, $(S(t): t \geq 0)$ is a Hilbertian $(C_{0},1)$-semigroup on $\Psi$.  Furthermore, the restriction $A$ of $-L$ to $\Psi$  is the infinitesimal generator of $(S(t): t \geq 0)$ on $\Psi$ and $A \in \mathcal{L}(\Psi,\Psi)$. In particular, we have for all $\psi \in \Psi$, $t \geq 0£$, 
\begin{equation}\label{eqConstrucGeneSemiCHNS}
 A\psi = -\sum_{j=1}^{\infty} \lambda_{j} \inner{\psi}{\phi_{j}}_{H} \phi_{j}, \quad S(t) \psi = \sum_{j=1}^{\infty} e^{-t\lambda_{j}} \inner{\psi}{\phi_{j}}_{H}  \phi_{j}.
\end{equation}
As an example, let $H=L^{2}(\R)$ and $-L=\frac{d^{2}}{dx^{2}}-\frac{x^{2}}{4}$. Consider the sequence of Hermite functions $(\phi_{n}: n \in \N)$ defined as:
$$ \phi_{n+1}(x)=\sqrt{g(x)} \, h_{n}(x), \quad n=0,1,2, \dots, $$
for $\displaystyle{g(x)=(\sqrt{2\pi})^{-1} \exp(-x^{2}/2)}$ and where $(h_{n}: n=0,1,2, \dots)$ is the sequence of Hermite polynomials:
$$h_{n}(x)=\frac{(-1)^{n}}{\sqrt{n!}} g(x)^{-1} \frac{d^{n}}{dx^{n}} g(x), \quad n=0,1,2, \dots. $$
We have $\sum_{n=1}^{\infty} \norm{(I+L)^{-r}\phi_{n}}_{H}^{2}= \sum_{n=1}^{\infty} \left( n+\frac{1}{2} \right)^{-2r}< \infty$ for $r > \frac{1}{2}$. Hence the operator  $(\lambda I+ L)^{-r_{1}}$ is Hilbert-Schmidt for $r_{1} > \frac{1}{2}$. Then from the above construction one can check (see Theorem 1.3.2 in \cite{KallianpurXiong}) that $\Psi=\mathscr{S}(\R)$ and the generator $A$ and equicontinuous semigroup $S(t)$ 
can be described by \eqref{eqConstrucGeneSemiCHNS} for $\lambda_{n}=n-\frac{1}{2}$ and the Hermite functions $\phi_{n}$. See \cite{KallianpurPerezAbreu:1988, KallianpurWolpert:1984, KallianpurXiong} for other examples using the above construction. 
\end{example}

\begin{example} Let $\Phi$ be a barrelled nuclear space. 
Recall from \cite{BojdeckiJakubowski:1990} that a $\Phi'$-valued adapted continuous zero-mean Gaussian process $W=( W_{t}: t \geq 0)$ is called a \emph{generalized Wiener process} if $W_{t}-W_{s}$ is independent of $\mathcal{F}_{s}$, for $0 \leq s < t$, 
and  
\begin{equation} \label{covarianceFunctionalGeneralizedWiener}
\Exp \left( \inner{W_{t}}{\phi} \inner{W_{s}}{\varphi} \right) = \int_{0}^{t \wedge s} q_{r}(\phi,\varphi)dr , \quad \forall \, t, s \in R_{+}, \, \phi \in \Phi. 
\end{equation} 
where $\{q_{r}: r \in \R_{+}\}$ is a family of continuous Hilbertian semi-norms on $\Phi$, such that the map $r \mapsto q_{r}(\phi,\varphi)$ is Borel measurable and bounded on finite intervals, for each $\phi$, $\varphi$ in $\Phi$. 

Consider the stochastic evolution equation
\begin{equation}\label{geneWienerSEE}
d X_{t}= (A'X_{t}+ B (t,X_{t})) dt+F(t,0,X_{t}) dW_{t}, \quad \mbox{for }t \geq 0,
\end{equation}
with initial condition $X_{0}=Z_{0}$. The process $W$ induces the cylindrical martigale valued measure $M(t,A)=W_{t} \delta_{0}(A)$ where \begin{inparaenum}[(i)] \item $U=\{ 0 \}$,  $\mathcal{R}=\mathcal{B}(\{ 0 \})$ and $\mu = \delta_{0}$, and \item $q_{r,0}=q_{r}$, where $\{q_{r}: r \in \R_{+}\}$ are as in \eqref{covarianceFunctionalGeneralizedWiener}. \end{inparaenum} Hence, in comparison with \eqref{generalFormSEE} we have $\int_{0}^{t} \int_{U} F(r,u,X_{r}) M (dr,du) = \int_{0}^{t} F(r,0,X_{r}) dW_{r}$. 

Suppose that $Z_{0}$, $A$, $(S(t): t \geq 0)$, $(S(t)': t \geq 0)$, $B$ and $F$ satisfy Assumption \ref{assumptionsCoefficients} and that $(S(t): t \geq 0)$ is a Hilbertian $(C_{0},1)$-semigroup on $\Psi$. By Theorem \ref{theoCadlagSolSEEwithCMVM} we have that \eqref{geneWienerSEE} has a unique weak solution with continuous paths and initial condition $X_{0}=Z_{0}$. Moreover, for every $T>0$ there exists a continuous Hilbertian seminorm $\rho=\rho(T)$ on $\Psi$ such that $X=( X_{t} : t \in [0,T])$ is a $\Psi'_{\rho}$-valued adapted continuous process satisfying $\sup_{t \in [0,T]} \Exp \left( \rho'(X_{t})^{2} \right)< \infty$. 
\end{example}

\begin{example}
Suppose $\Phi$ is a barrelled nuclear space and consider a $\Phi'$-valued L\'{e}vy process $L=( L_{t} : t\geq 0)$ (see Section 3.2 in \cite{FonsecaMora:Levy}). 
It is shown in (\cite{FonsecaMora:Levy}, Theorem 4.17) that the L\'{e}vy process $L$ admits a \emph{L\'{e}vy-It\^{o} decomposition}, i.e.  for each $t \geq 0$, 
\begin{equation} \label{eqLevyItoDecomposition}
L_{t}=t\goth{m}+W_{t}+\int_{B_{\rho'}(1)} f \widetilde{N} (t,df)+\int_{B_{\rho'}(1)^{c}} f N (t,df).
\end{equation}
In \eqref{eqLevyItoDecomposition}, we have that $\goth{m} \in \Phi'$, $( W_{t} : t \geq 0)$ is a  $\Phi'$-valued Wiener process with zero-mean and \emph{covariance functional} $\mathcal{Q}$ (see Section 3.4 in \cite{FonsecaMora:Levy}).  
Moreover, $ N(t,A)= \sum_{0 \leq s \leq t} \ind{A}{\Delta L_{s}}$, $\forall \, t \geq 0$, $A \in \mathcal{B}( \Phi' \setminus \{ 0\})$, is the \emph{Poisson random measure} associated to $L$ with respect to the ring $\mathcal{A}$ of all the subsets of $\Phi' \setminus \{0\}$ that are \emph{bounded below} (i.e. $A \in \mathcal{A}$ if $0 \notin \overline{A}$), $\widetilde{N}(dt,df)= N(dt,df)-dt \, \nu(df)$  where $\nu$ is a \emph{L\'{e}vy measure} on $\Phi'$ (see Section 4.3 in \cite{FonsecaMora:Levy}) with a continuous Hilbertian semi-norm $\rho$ on $\Phi$ such that $\int_{B_{\rho'}(1)} \rho'(f)^{2} \nu (df) < \infty$. 
The process $\int_{B_{\rho'}(1)} f \widetilde{N} (t,df)$, $t\geq 0$, is a $\Phi'$-valued zero-mean, square integrable, c\`{a}dl\`{a}g L\'{e}vy process, and the process $\int_{B_{\rho'}(1)^{c}} f N (t,df)$  $\forall t\geq 0$ is a $\Phi'$-valued c\`{a}dl\`{a}g L\'{e}vy process defined  by means of a Poisson integral (see Section 4.1 in \cite{FonsecaMora:Levy}) with respect to the Poisson random measure $N$ of $L$ on the set $B_{\rho'}(1)^{c}$. All the random components of the representation \eqref{eqLevyItoDecomposition} are independent. 

Consider the following 
\emph{L\'{e}vy-driven stochastic evolution equation}:
\begin{multline}\label{levyDrivenSEELID}
 d X_{t}= (A'X_{t}+ B (t,X_{t})) dt+ F(t,0,X_{t}) dW_{t} \\
+ \int_{B_{\rho'}(1)} F(t,u,X_{t}) \tilde{N}(dt,du)+ \int_{B_{\rho'}(1)^{c}} F(t,u,X_{t}) N(dt,df). 
\end{multline}
Suppose that $Z_{0}$, $A$, $(S(t): t \geq 0)$, $(S(t)': t \geq 0)$, $B$ and $F$ satisfy Assumption \ref{assumptionsCoefficients}  for $U= \Phi'$, $\mu=\nu$, and with the family of continuous Hilbertian semi-norms  $\{ q_{r,u}: r \in \R_{+}, u \in \Phi' \}$ given by $q_{r,0}(\phi)= \mathcal{Q}(\phi,\phi)^{1/2}$ and $q_{r,u}(\phi) =\abs{\inner{u}{\phi}}$ if $u \neq 0$.

We know by Theorem 7.1 in \cite{FonsecaMora:2018-1} that \eqref{levyDrivenSEELID} 
has a weak and mild solution. If we further assume that  $(S(t): t \geq 0)$ is a Hilbertian $(C_{0},1)$-semigroup on $\Psi$, then if in the proof of Theorem 7.1 in \cite{FonsecaMora:2018-1} we use Theorem \ref{theoCadlagSolSEEwithCMVM} instead of Theorem 6.23 in \cite{FonsecaMora:2018-1}, we can show that \eqref{levyDrivenSEELID} has a unique c\`{a}dl\`{a}g weak solution with initial condition $X_{0}=Z_{0}$.
\end{example}

\smallskip

\noindent \textbf{Acknowledgements} The author thanks The University of Costa Rica for providing financial support through the grant ``821-C2-132- Procesos cil\'{i}ndricos y ecuaciones diferenciales estoc\'{a}sticas''.


\end{document}